\documentclass{amsart}%kh
\usepackage{amssymb}
\usepackage{amsmath}
\usepackage{latexsym}
\usepackage{eepic}
\usepackage{epsfig}
\usepackage{graphicx}
\usepackage{amscd}
\usepackage{eufrak}
\usepackage{mathrsfs}
\usepackage[english]{babel}
\usepackage[dvips]{color}

\newtheorem{theorem}[equation]{Theorem}
\newtheorem{theorem-definition}[equation]{Theorem-Definition}
\newtheorem{lemma-definition}[equation]{Lemma-Definition}
\newtheorem{definition-prop}[equation]{Proposition-Definition}

\newtheorem{prop}[equation]{Proposition}

\newtheorem{cor}[equation]{Corollary}

\theoremstyle{definition}
\newtheorem{exam}[equation]{Example}

\newcommand{\llbr}{[\negthinspace[}
\newcommand{\rrbr}{]\negthinspace]}

\theoremstyle{definition}
\newtheorem{example}[equation]{Example}

\newcommand{\N}{\ensuremath{\mathbb{N}}}
\newcommand{\Z}{\ensuremath{\mathbb{Z}}}
\newcommand{\Q}{\ensuremath{\mathbb{Q}}}

\newcommand{\R}{\ensuremath{\mathbb{R}}}
\newcommand{\C}{\ensuremath{\mathbb{C}}}

\newcommand{\A}{\ensuremath{\mathbb{A}}}

\newcommand{\G}{\ensuremath{\mathbb{G}}}

\newcommand{\cX}{\ensuremath{\mathscr{X}}}

\newcommand{\cH}{\ensuremath{\mathscr{H}}}

\renewcommand{\R}{\ensuremath{\mathbb{R}}}
\renewcommand{\C}{\ensuremath{\mathbb{C}}}

\renewcommand{\A}{\ensuremath{\mathbb{A}}}

\renewcommand{\cH}{\ensuremath{\mathscr{H}}}

\newcommand{\Spec}{\ensuremath{\mathrm{Spec}\,}}

\newcommand{\red}{\mathrm{red}}

\newcommand{\weight}{\mathrm{wt}}

\newcommand{\lct}{\mathrm{lct}}
\newcommand{\Sk}{\mathrm{Sk}}

\newcommand{\an}{\mathrm{an}}
\newcommand{\bir}{\mathrm{bir}}
\newcommand{\divi}{\mathrm{div}}
\newcommand{\mon}{\mathrm{mon}}
\newcommand{\spe}{\mathrm{sp}}
\newcommand{\sncd}{\mathrm{sncd}}

\numberwithin{equation}{subsection}

\newcommand{\sss}{\vspace{5pt} \subsubsection*{ }\refstepcounter{equation}{{\bfseries(\theequation)}\ }}

\author[Johannes Nicaise]{Johannes Nicaise}
\address{KU Leuven\\
Department of Mathematics\\
Celestijnenlaan 200B\\3001 Heverlee \\
Belgium} \email{johannes.nicaise@wis.kuleuven.be}

\begin{document}
\title{Berkovich skeleta and birational geometry}

\begin{abstract}
  We give a survey of joint work with Mircea Musta\c{t}\u{a} and Chenyang Xu on the connections between the geometry of Berkovich spaces over the field of Laurent series and the birational geometry of one-parameter degenerations of smooth projective varieties. The central objects in our theory are the {\em weight function} and the {\em essential skeleton} of the degeneration. We tried to keep the text self-contained, so that it can serve as an introduction to Berkovich geometry for birational geometers.
\end{abstract}
\maketitle

\section{Introduction}
 Let $R$ be a complete discrete valuation ring with residue field $k$ and quotient field $K$.
    The main example to keep in mind is
  $R=\C \llbr t \rrbr$. The discrete valuation on $K$ gives rise to a non-archimedean absolute value on $K$ that
   one can use to develop a theory of analytic geometry over $K$. The theory that we will use is the one introduced by Berkovich in \cite{berk-book}.
  The principal purpose of these notes is to describe some interactions between Berkovich geometry over $K$ and the birational geometry of degenerations of algebraic varieties over $R$. For a nice introduction to related results over trivially valued base fields, we refer to \cite{payne}.

   In fact, we will use only a small part of the theory of Berkovich spaces: we are mainly interested in the underlying topological space of the analytification of
    an algebraic $K$-variety. The structure of this space can be described in terms of classical valuation theory; this will be explained in
    Section \ref{sec-berksk}.

\medskip

Let $X$ be a connected, smooth and proper $K$-variety of dimension $n$. We denote by $X^{\an}$ the Berkovich analytification of $K$.
 An $sncd$-model of $X$ is a regular scheme $\cX$ of finite type  over $R$, endowed with an isomorphism of $K$-schemes $\cX_K\to X$, such that the special fiber
  $\cX_k$ is a divisor with strict normal crossings.
 To any proper $sncd$-model $\cX$
 of $X$ over $R$ one can attach a subspace $\Sk(\cX)$ of $X^{\an}$, called the Berkovich skeleton of $X$, which is canonically homeomorphic to the dual
 intersection complex of the strict normal crossings divisor $\cX_k$. This skeleton can be viewed as the space of real valuations on the function field of $X$ that extend the discrete valuation on $K$ and that are monomial with respect to $\cX_k$. The most important property of the skeleton $\Sk(\cX)$ is that it controls the homotopy type of $X^{\an}$: it is a strong deformation retract of $X^{\an}$. This provides an interesting link between the geometry of $X^{\an}$ and the birational geometry of models of $X$.

 If $X$ has dimension one and genus at least one, then $X$ has a unique minimal $sncd$-model, and thus a canonical Berkovich skeleton.
  In higher dimensions, minimal $sncd$-models no longer exist. Nevertheless, one can ask whether it is still possible to construct a canonical skeleton in $X^{\an}$.
  We will present two constructions, which we developed in collaboration with Mircea Musta\c{t}\u{a} \cite{MuNi} and Chenyang Xu \cite{NiXu}, respectively. The first construction is based on work of Kontsevich and Soibelman on degenerations of Calabi-Yau varieties and Mirror Symmetry \cite{KS}; the second one relies on the Minimal Model Program, and in particular on the results in \cite{dFKX}. As we will see, both approaches yield the same result. We assume in the remainder of this introduction that the residue field $k$ has characteristic zero.

 The main idea behind the first approach is the following. Each proper $sncd$-model $\cX$ of $X$ gives rise to a Berkovich skeleton $\Sk(\cX)$ in $X^{\an}$.
 We will use pluricanonical forms $\omega$ on $X$ to single out certain essential faces of the simplicial complex $\Sk(\cX)$, which must be contained in the skeleton
  of {\em every} proper $sncd$-model of $X$. The union of these $\omega$-essential faces is called the Kontsevich-Soibelman skeleton of $(X,\omega)$ and denoted by $\Sk(X,\omega)$. It only depends on $X$ and $\omega$, but not on the choice of $\cX$. Taking the union of the skeleta $\Sk(X,\omega)$ over all non-zero pluricanonical forms $\omega$ on $X$, we obtain a subspace of $X^{\an}$ with piecewise affine structure that we call the {\em essential skeleton} of $X$ and that we denote by $\Sk(X)$.

 This construction has the merit of being quite natural and elementary, but it is not at all clear from the definition what $\Sk(X)$ looks like or whether $\Sk(X)$ is still a strong deformation retract of $X^{\an}$. Therefore, we will also consider a second approach. As we have already mentioned, minimal $sncd$-models usually do not exist if the dimension of $X$ is at least two, but we can enlarge our class of models in such a way that minimal models exist and such that we can still use the members of this class to describe the homotopy type of $X^{\an}$. The Minimal Model Program suggests to consider so-called $dlt$-models of $X$, which should be viewed as proper $sncd$-models with mild singularities. We can define the skeleton $\Sk(\cX)$ of such a $dlt$-model by simply ignoring the singularities. The theory of minimal models guarantees that minimal $dlt$-models exist if the canonical sheaf of $X$ is semi-ample, which means that some tensor power is generated by global sections. Minimal $dlt$-models are not unique, but the skeleton $\Sk(\cX)$ does not depend on the choice of a minimal $dlt$-model $\cX$. By a careful analysis of the steps in the Minimal Model Program, it was proven in \cite{dFKX} that the skeleton $\Sk(\cX)$ can be obtained from the skeleton of any proper $sncd$-model of $X$ by a sequence of elementary collapses. This implies that $\Sk(\cX)$ is still a strong deformation retract of $X^{\an}$.

 One of the main results of \cite{NiXu} is that these two constructions yield the same result: if the canonical sheaf of $X$ is semi-ample, then the essential skeleton $\Sk(X)$ of $X$ coincides with the skeleton of any minimal $dlt$-model of $X$. In particular, $\Sk(X)$ is a strong deformation retract of $X^{\an}$. The semi-ampleness condition can be understood as follows: it guarantees that $X$ has enough pluricanonical forms to detect all the important pieces of the skeleton of a proper $sncd$-model.

\subsection*{Acknowledgements} I am grateful to Sam Payne for helpful comments on an earlier version of this text.

\subsection*{Notation} We denote by
$R$ a complete discrete valuation ring with maximal ideal
$\frak{m}$, residue field $k$ and quotient field $K$. We denote by
$v_K$ the discrete valuation $K^{\ast}\twoheadrightarrow \Z$. We define
an absolute value on $K$ by setting $|x|_K=\exp(-v_K(x))$ for every
element $x$ of $K^*$. A variety over a field $F$ is a separated $F$-scheme of finite type.
 If $\alpha$ and $\beta$ are elements of $\R^m$ for some positive integer $m$, then we denote by $\alpha\cdot \beta$ their scalar product $\sum_{i=1}^m\alpha_i\beta_i$.

 \newpage

\section{The Berkovich skeleton of an $sncd$-model}\label{sec-berksk}
\subsection{Birational  points}
\sss \label{sss-defan}
Let $X$ be a connected and smooth $K$-variety of dimension $n$.
 We denote by $X^{\an}$ the Berkovich analytification of $X$
and by $i\colon X^{\an}\to X$ the analytification morphism. We will mainly be interested in the underlying topological space of $X^{\an}$, which is
easy to describe. As a set, $X^{\an}$ consists of the couples $(x,|\cdot|)$ where $x$ is a scheme-theoretic point of $X$ and $|\cdot|\colon\kappa(x)\to \R$ is an absolute value on
 the residue field $\kappa(x)$ of $X$ at $x$ that extends the absolute value  $|\cdot|_K$ on $K$. The analytification map $i\colon X^{\an}\to X$ is simply the forgetful map that sends
 a couple $(x,|\cdot|)$ to $x$. The topology on $X$ is the coarsest topology such that the following two properties are satisfied:
 \begin{enumerate}
 \item the topology on $X^{\an}$ is finer than the Zariski topology, that is, the map $i\colon X^{\an}\to X$ is continuous;
 \item for every Zariski-open subset $U$ of $X$ and every regular function $f$ on $U$, the map
 $$|f|\colon i^{-1}(U)\to \R^+\colon (x,|\cdot|)\mapsto |f(x)|$$ is continuous.
 \end{enumerate}
Note that the definition of $|f|$ makes sense because $f(x)$ is an element of the residue field $\kappa(x)$.
 We will often denote a point of $X^{\an}$ simply by $x$, leaving the absolute value $|\cdot|$ implicit in the notation. It is convenient in many situations
 to switch between the multiplicative and additive viewpoint: we will denote by $v_x$ the real valuation
 $$v_x\colon \kappa(x)^{\ast}\to \R\colon a\mapsto -\ln |a|.$$ As usual, we extend it to zero by setting $v_x(0)=+\infty$.
 The {\em residue field} of $X^{\an}$ at a point $x$ is defined as the completion of the residue field $\kappa(i(x))$ of $X$ at $i(x)$ with respect to the absolute value $|\cdot|$.
 It is a complete valued extension of the field $K$, which we denote by $\cH(x)$. The valuation ring of $\cH(x)$ will be denoted by $\cH(x)^o$.
  If $f$ is a rational function on $X$ that is defined at $i(x)$, then we can think of $f(i(x))\in \kappa(i(x))$ as an element of $\cH(x)$, and we denote this element by $f(x)$.

\sss The topological space $X^{\an}$ is Hausdorff, and it is compact if and only if $X$ is proper over $K$.
 If $X$ is a curve, then there exists a simple classification of the points on $X^{\an}$ and one can draw a fairly explicit picture of $X^{\an}$; see for instance Section 5.1 in \cite{baker}.
    If the dimension of $X$ is at least two, it is much more difficult to give a precise description of the whole space $X^{\an}$. We will see in the following sections, however, that one can produce many interesting points
in this space using the birational geometry of $X$, and that these points suffice to control the homotopy type of $X^{\an}$.

\sss The set
$X^{\bir}$ of birational points of $X^{\an}$ is defined as
the inverse image under $i\colon X^{\an}\to X$ of the generic point of $X$. In the additive notation, this is simply
 the set of real valuations on the
function field $K(X)$ of $X$ that extend the discrete valuation $v_K$ on
$K$.  We endow $X^{\bir}$
with the topology induced by the Berkovich topology on $X^{\an}$. We will see in \eqref{sss-defret} that the inclusion $X^{\bir}\to X^{\an}$ is a homotopy equivalence if $k$ has characteristic zero. By its very definition, $X^{\bir}$ is a birational invariant of $X$, so we can hope
 to recover interesting birational invariants of $X$ from this topological space.

\subsection{Models}
\sss We will define certain subclasses of birational points using the geometry of $R$-models of $X$.
 An $R$-model of $X$ is a flat separated $R$-scheme of finite
type $\cX$ endowed with an isomorphism of $K$-schemes $\cX_K\to
X$. Note that we do not impose any properness condition on $X$ or $\cX$. We say that $\cX$ is an $sncd$-model of $X$ if $\cX$ is regular and its special fiber $\cX_k$ is a divisor with strict normal crossings.

\sss Let $\cX$ be an $R$-model of $X$ and let $x$ be a point of $X^{\an}$. We say that $x$ has a {\em center} on $\cX$
 if the canonical morphism $\Spec \cH(x)\to X$ extends to a morphism $\Spec \cH(x)^o\to \cX$. Such an extension is unique if it exists, by
 the valuative criterion of separatedness. If it exists, the center of $x$ on $\cX$ is defined
 as the image of the closed point of $\Spec \cH(x)^o$ in $\cX$ and denoted by $\spe_{\cX}(x)$. Note that the point $\spe_{\cX}(x)$ always lies on the special fiber $\cX_k$ of $\cX$ because $\frak{m}$ is contained in the maximal ideal of $\cH(x)^o$. We denote by $\widehat{\cX}_\eta$ the set of points on $X^{\an}$
 that have a center on $\cX$.
  If $\cX$ is proper over $R$, then $\widehat{\cX}_\eta=X^{\an}$ by the valuative criterion of properness.
  The map $$\spe_{\cX}\colon \widehat{\cX}_\eta\to \cX_k$$ is called the reduction map or specialization map. It has the peculiar property of being {\em anti-continuous}, which means that the inverse image of an open set is closed.

\begin{exam}
If $X=\A^1_K=\Spec K[T]$ and $\cX=\A^1_R$, then
$$\widehat{\cX}_\eta=\{x\in X^{\an}\,|\,|T(x)|\leq 1\}$$ and
 $\spe_{\cX}(x)$ is the reduction of
 $T(x)\in \cH(x)^o$ modulo the maximal ideal of $\cH(x)^o$ (viewed as a point of $\cX_k=\Spec k[T]$).
\end{exam}

\sss Assuming a bit more technology \cite[\S1]{berk-vanish}, we can give an equivalent description of
$\widehat{\cX}_\eta$ and $\spe_{\cX}$. If we denote by $\widehat{\cX}$ the
 formal $\frak{m}$-adic completion of $\cX$, then the
generic fiber  of $\widehat{\cX}$ is a compact
 analytic domain in $X^{\an}$ whose underlying set is precisely $\widehat{\cX}_\eta$.
 The reduction map $\spe_{\cX}$ is the map underlying the specialization morphism of locally ringed spaces
$$\spe_{\cX}\colon \widehat{\cX}_\eta\to \widehat{\cX}.$$

\subsection{Divisorial and monomial points}\label{subsec-divmon}

\sss  Let $\cX$ be a normal $R$-model of $X$. If $E$ is an irreducible component of the special fiber $\cX_k$
 with generic point $\xi$, then the fiber $\spe^{-1}_{\cX}(\xi)$
 consists of a unique point, which we call the divisorial
 point of $X^{\an}$ associated with $(\cX,E)$. It is the birational point $x$ on $X^{\an}$ that corresponds to the discrete
 valuation $v_x$ on $K(X)$ with valuation ring $\mathcal{O}_{\cX,\xi}$, normalized in such a way
 that $v_x$ extends the discrete valuation $v_K$ on $K$. Thus if $f$ is a non-zero rational function on $X$, then
 $$v_x(f)=\frac{1}{N}\mathrm{ord}_E f$$ where $N$ denotes the multiplicity of $E$ in the Cartier divisor $\cX_k$ on $\cX$ and $\mathrm{ord}_E f$ is the order of $f$ along $E$.
  A point of $X^{\an}$ is called divisorial if it is the divisorial point associated with some couple $(\cX,E)$ as above.
 We will
 denote the set of divisorial points by $X^{\divi}\subset
 X^{\bir}$. It is not difficult to show that the set $X^{\divi}$ is dense in $X^{\an}$; see for instance \cite[2.4.12]{MuNi}.

\sss The set of divisorial points is totally disconnected. We will define a more general class of points that should be viewed as
 some kind of interpolations between divisorial points: the monomial points.
  Let $\cX$ be an $sncd$-model of $X$ and let $E_1,\ldots,E_r$ be distinct irreducible components of the special fiber $\cX_k$ with respective
  multiplicities $N_1,\ldots,N_r$ in $\cX_k$, and assume that the intersection $\cap_{i=1}^rE_i$ is non-empty. Let $\alpha=(\alpha_1,\ldots,\alpha_r)$ be a tuple of positive real numbers such that $\sum_{i=1}^r\alpha_iN_i=1$ and let $\xi$ be a generic point of  $\cap_{i=1}^rE_i$ (by the definition of an $sncd$-model, this intersection is regular and of pure dimension $n+1-r$, but it is not necessarily connected).

 \begin{prop}\label{prop-mon}
 There exists a unique minimal real valuation
 $$v\colon \mathcal{O}_{\cX,\xi}\setminus \{0\}\to \R^+$$ such that $v(T_i)=\alpha_i$ for every $i$ in $\{1,\ldots,r\}$ and every local equation
 $T_i=0$ for $E_i$  in $\cX$ at $\xi$.
 \end{prop}

 \sss \label{sss-monexpl} Proposition \ref{prop-mon} can be proven by combining \cite[2.4.6 and 3.1.6]{MuNi}. We will not give a complete proof here, but only sketch how the valuation $v$ can be constructed. For every $i$ in $\{1,\ldots,r\}$ we choose a local equation $T_i=0$ of $E_i$ in $\cX$ at $\xi$. Then the elements $T_i$ form a regular system of local parameters in the local ring $\mathcal{O}_{\cX,\xi}$. It is not difficult to show that every element $f$ in $\mathcal{O}_{\cX,\xi}$ can be written in the completed local ring $\widehat{\mathcal{O}}_{\cX,\xi}$
  as a power series
  \begin{equation}\label{eq-adm}f=\sum_{\beta\in \N^r}c_\beta T^\beta \end{equation}
   where each coefficient $c_\beta$ is either zero or a unit in $\widehat{\mathcal{O}}_{\cX,\xi}$. Such an expansion is not unique, but one can show that the expression
  \begin{equation}\label{eq-monexpl0}
  v(f):=\min\{\alpha\cdot \beta\,|\,\beta\in \N^r,\,c_\beta\neq 0\} \end{equation}
   does not depend on any choices and that it defines a valuation $v$ with the required properties.
   If $R$ has equal characteristic, then the arguments can be simplified by using the fact that $\widehat{\mathcal{O}}_{\cX,\xi}$ is isomorphic to the power series ring $\kappa(\xi)\llbr T_1,\ldots,T_r\rrbr$ by Cohen's structure theorem.

 \sss The valuation $v$ in Proposition \ref{prop-mon} extends to a real valuation $v\colon K(X)^{\ast}\to \R$.
  It extends the discrete valuation $v_K$ on $K$: if $\pi$ is a uniformizer in $R$, then in the ring $\mathcal{O}_{\cX,\xi}$
 we can write $$\pi=u\prod_{i=1}^r T_i^{N_i}$$ with $u$ a unit,  so that $v(\pi)=\sum_{i=1}^r\alpha_iN_i=1$.
  Thus $v$ defines a birational point $x$ on $X^{\an}$, which we call the
  monomial point associated with the data \begin{equation}\label{eq-mondata}
  (\cX,(E_1,\ldots,E_r),\alpha,\xi).
   \end{equation} The point $x$ belongs to $\widehat{\cX}_\eta$, and $\spe_{\cX}(x)=\xi$.
 We remark for later use that formula \eqref{eq-monexpl0} can be generalized as follows: if
  $$f=\sum_{\beta\in \N^r}c_\beta d_\beta T^\beta$$ where each coefficient $c_\beta$ is either zero or a unit in $\widehat{\mathcal{O}}_{\cX,\xi}$ and each coefficient $d_\beta$ belongs to $K$, then
  \begin{equation}\label{eq-monexpl}
  v(f)=\min\{v_K(d_\beta)+\alpha\cdot \beta\,|\,\beta\in \N^r,\,c_\beta\neq 0\}
  \end{equation}
   since we can rewrite $d_\beta$ as  the product of  $$\pi^{v_K(d_\beta)}=(u\prod_{i=1}^r T_i^{N_i})^{v_K(d_\beta)}$$ with a unit in $R$ to get an expansion for $f$ of the form
   \eqref{eq-adm}.

  \sss     A point on $X^{\an}$ is called monomial if it is the monomial point associated with a tuple of data as in \eqref{eq-mondata}; we will
   also say that such a point is monomial with respect to the model $\cX$. If $r=1$, then we get precisely the divisorial point associated with $(\cX,E_1)$. Thus every divisorial point is monomial. Conversely, the monomial point associated with \eqref{eq-mondata} is divisorial (possibly with respect to
    a different model $\cX$) if and only if
   the parameters $\alpha_i$ all belong to $\Q$ (see \cite[2.4.1]{MuNi}). The set of monomial points on $X^{\an}$ will be denoted by $X^{\mon}$. We have the following inclusions:
   $$X^{\divi}\subset X^{\mon}\subset X^{\bir}\subset X^{\an}.$$

\subsection{The Berkovich skeleton}\label{subsec-berksk}
\sss Let $\cX$ be an $sncd$-model of $X$. We define the Berkovich skeleton of $\cX$ as the set of all points of $X^{\an}$ that are monomial with respect to $\cX$,
 and we denote it by $\Sk(\cX)$. By construction, the Berkovich skeleton $\Sk(\cX)$ is a subspace of $\widehat{\cX}_\eta\cap X^{\mon}$. The importance of this object is that
 we can give an explicit description of the topology on $\Sk(\cX)$ and that this suffices to understand the homotopy type of $\widehat{\cX}_\eta$, as we will now explain.

\sss We first need to recall the definition of the {\em dual complex} of the strict normal crossings divisor $\cX_k$.
 We write $\cX_k=\sum_{i\in I}N_i E_i$ and for every non-empty subset $J$ of $I$, we set $E_J=\cap_{i\in J} E_i$.
  The dual complex of $\cX_k$ is a simplicial complex\footnote{To be precise, $|\Delta(\cX_k)|$ is not a simplicial complex in the strict sense, because we allow for instance multiple edges between two vertices. This has no importance for the present exposition (and can always be remediated by blowing up $\cX$ at connected components of the subvarieties $E_J$, which gives rise to a stellar subdivision of the corresponding face of $|\Delta(\cX_k)|$). In any case, $|\Delta(\cX_k)|$ is the topological realization of a finite simplicial set.} $|\Delta(\cX_k)|$ whose simplices of dimension $d$ correspond bijectively to
 the connected components of the regular $k$-varieties $E_J$ where $J$ runs through the set of subsets of $I$ of cardinality $d+1$. If $J$ and $J'$ are non-empty subsets of $I$,
 and $C$ and $C'$ are connected components of $E_J$ and $E_{J'}$, respectively, then the simplex corresponding to $C$ is a face of the simplex corresponding to $C'$ if and only if
 $C$ contains $C'$. Thus the vertices of $|\Delta(\cX_k)|$ correspond to the irreducible components $E_i$ of $\cX_k$,
   and we will denote the vertices accordingly by $\nu_i$, $i\in I$. If $i$ and $j$ are distinct elements of $I$ then the number of edges between $\nu_i$ and $\nu_j$ is the number of connected components of $E_i\cap E_j$, and so on.
 In this way, the dual complex $|\Delta(\cX_k)|$ encodes the combinatorial structure of the intersections of prime components in $\cX_k$. The dimension of $|\Delta(\cX_k)|$ is at most $n$, the dimension of $X$.
  If $X$ has dimension one, then the dual complex $|\Delta(\cX_k)|$ is more commonly known as the {\em dual graph} of the special fiber $\cX_k$.

  \begin{example}
  Assume that $X$ has dimension one and that $\cX_k$ has four irreducible components $E_1,E_2,E_3,E_4$ such that $E_1$ intersects each of the other components in precisely one point and there are no other intersection points. Then $|\Delta(\cX_k)|$ is a graph with four vertices $v_1,v_2,v_3,v_4$ with one edge between $v_1$ and $v_i$ for $i=2,3,4$ and no other edges.

 If $X$ has dimension two and $\cX_k$ is isomorphic to the union of the coordinate planes in $\mathbb{A}^3_k$, then $|\Delta(\cX_k)|$ is the standard 2-simplex.
  \end{example}

 \sss \label{sss-phi} We will now construct a map $$\Phi\colon |\Delta(\cX_k)|\to \Sk(\cX).$$ For each $i\in I$, our map $\Phi$ sends the vertex $\nu_i$ of $|\Delta(\cX_k)|$ to the divisorial point associated with $(\cX,E_i)$. In order to define $\Phi$ on the higher-dimensional faces
  of $|\Delta(\cX_k)|$, we use monomial valuations to interpolate between these divisorial valuations, as follows.
   Let $y$ be a point of  $|\Delta(\cX_k)|$. Then there exists a unique face $\tau$ of $|\Delta(\cX_k)|$ such that $y$ lies in the interior $\tau^o$ of $\tau$.
  By the construction of $|\Delta(\cX_k)|$, the face $\tau$ corresponds to a connected component $C$ of an intersection $E_J$ for some subset $J$ of $I$. We denote by $\xi$ the generic point of $C$. The vertices of $\tau$
  correspond precisely to the irreducible components $E_i$ with $i\in J$.  We can represent the point $y$ by a tuple
  of barycentric coordinates $\beta\in \R^J$ where each coordinate $\beta_i$ is a positive real number and their sum is equal to one. Now we define $\Phi(y)$ as the
  monomial point of $X^{\an}$ associated with the data
  $$(\cX,(E_i)_{i\in J},(\beta_i/N_i)_{i\in J},\xi).$$

  \sss It is easy to see that $\Phi$ is a bijection, since we can give the following description of the map $\Phi^{-1}$. Let $x$ be a point of $\Sk(\cX)$.
   Then $\spe_{\cX}(x)$ is a generic point $\xi$ of $E_J$, for some uniquely determined non-empty subset $J$ of $I$. For each $i\in J$ we choose a local equation $T_i=0$ for
   $E_i$ in $\cX$ at $\xi$, and we set $\alpha_i=v_x(T_i)$. Then $\Phi^{-1}(x)$ lies in the interior of the face $\tau$ of $|\Delta(\cX_k)|$ corresponding to the
   connected component of $\xi$ in $E_J$, and its tuple of barycentric coordinates  is equal to $(\alpha_iN_i)_{i\in J}$.
     In fact, we can say more.

   \begin{prop}\label{prop-homeo}
   The map
   $$\Phi\colon |\Delta(\cX_k)|\to \Sk(\cX)$$ is a homeomorphism.
   \end{prop}
   \begin{proof}
   Since the source of $\Phi$ is compact and the target is Hausdorff, we only need to prove that $\Phi$ is continuous.
    This is not difficult: using the definition of the Berkovich topology in \eqref{sss-defan} and the explicit description of monomial valuations in \eqref{sss-monexpl},
 one immediately checks that $\Phi$ is continuous on the interior of each of the faces of $|\Delta(\cX_k)|$. To get the continuity at the boundary faces,
  one uses the following easy observation. Suppose that some of the $\alpha_i$ are zero in the construction of the valuation $v$ in \eqref{sss-monexpl},
  say, $\alpha_1,\ldots,\alpha_s\neq 0$ and $\alpha_{s+1}=\ldots=\alpha_r=0$ for some $s<r$. Then the formula we gave in \eqref{eq-monexpl0} defines the monomial valuation
   associated with $\cX$, the components $E_1,\ldots,E_s$, the parameters $\alpha_1,\ldots,\alpha_s$ and the unique generic point of $E_1\cap\ldots\cap E_s$ whose closure
   contains $\xi$.  For details, see \cite[2.4.9 and 3.1.4]{MuNi}.
   \end{proof}

 \sss We can use the homeomorphism $\Phi$ to endow $\Sk(\cX)$ with a piecewise $\Z$-affine structure; see \cite[\S3.2]{MuNi}.
    This structure
  can be defined intrinsically on $X^{\an}$ and is independent of the choice of the model $\cX$.  The induced piecewise $\Q$-affine structure is simply the one inherited from the faces of
 $|\Delta(\cX)|$. We will not use the finer $\Z$-affine structure so we will not recall its definition here.  If $f$ is a non-zero rational function on $X$, then the function
$$\Sk(\cX)\to \R \colon x \mapsto \ln|f(x)|$$
 is continuous and piecewise affine.

 \sss Proposition \ref{prop-homeo} gives an explicit description of the topological space $\Sk(\cX)$. We will now explain how one can use this description
  to determine the homotopy type of $\widehat{\cX}_\eta$. First, we construct a retraction $$\rho_{\cX}\colon \widehat{\cX}_\eta\to \Sk(\cX)$$ for the
  embedding of $\Sk(\cX)$  in $\widehat{\cX}_\eta$. Let $x$ be a point of $\widehat{\cX}_\eta$.
   Let $J$ be the set of indices $i\in I$ such that $E_i$ contains  the center $\spe_{\cX}(x)$ of $x$ on $\cX$. We denote by $C$ the connected component of $x$ in  $E_J$  and by $\xi$ the generic point of $C$. For each $i\in J$ we choose a local equation $T_i=0$ for
   $E_i$ in $\cX$ at $\spe_{\cX}(x)$, and we set $\alpha_i=v_x(T_i)$. Then $\rho_{\cX}(x)$ is the monomial point in $X^{\an}$ associated with the data
   $$(\cX,(E_i)_{i\in J},(\alpha_i)_{i\in J},\xi).$$ In other words, it is the unique point of the skeleton such that the Zariski closure of its center contains the center of $x$
   and which gives the same valuation to each local defining equation of an irreducible component $E_i$ of $\cX_k$ passing through $\spe_{\cX}(x)$. It is an easy exercise to
   verify that $\rho_{\cX}$ is continuous. The most fundamental result about Berkovich skeleta is the following theorem.

   \begin{theorem}[Berkovich, Thuillier]\label{thm-sk}
   There exists a continuous map
   $$H:\widehat{\cX}_\eta \times [0,1]\to \widehat{\cX}_\eta$$ such that
   $H(\cdot,0)$ is the identity, $H(\cdot,1)$ is the map $\rho_{\cX}\colon \widehat{\cX}_\eta\to \Sk(\cX)$, and $H(x,t)=x$ for every point $x$ of $\Sk(\cX)$
    and every $t$ in $[0,1]$. Thus $\Sk(\cX)$ is a strong deformation retract of $\widehat{\cX}_\eta$.
   \end{theorem}
   \begin{cor}\label{cor-sk}
   If $X$ is proper and $\cX$ is a proper $sncd$-model of $X$, then $\Sk(\cX)$ is a strong deformation retract of $X^{\an}$. In particular, $X^{\an}$ has the same homotopy type
   as the simplicial complex $|\Delta(\cX_k)|$.
   \end{cor}
   \begin{proof} This follows from the fact that $\widehat{\cX}_\eta=X^{\an}$ if $\cX$ is proper, and from Proposition \ref{prop-homeo}.
   \end{proof}

   \sss  Giving a proof of Theorem \ref{thm-sk} goes beyond the scope of this survey, but we will work out an elementary example in Section
   \ref{subsec-exam}. The origins of Theorem \ref{thm-sk} are the results by Berkovich on skeleta of so-called poly-stable formal schemes \cite{berk-contr}.
   Berkovich used these skeleta to prove that smooth non-archimedean analytic spaces are locally contractible. An $sncd$-model $\cX$ (or rather, its formal $\frak{m}$-adic completion) is not poly-stable unless the special fiber $\cX_k$ is reduced\footnote{However, the class of poly-stable formal schemes is much larger than the class of $sncd$-models with reduced special fiber.}. Thus we cannot directly apply Berkovich's result here. If $R$ has equal characteristic and $\cX$ is defined over an algebraic curve,  we explained in \cite[3.1.3]{NiXu} how one can deduce  Theorem \ref{thm-sk} from results by Thuillier on skeleta over trivially valued fields \cite{thuillier}. The general case can be proven by translating Thuillier's toroidal methods into the language of log-geometry; details will be given in a forthcoming publication. 

   \sss If $X$ is proper over $K$, then the existence of a proper $sncd$-model $\cX$ is known if $k$ has characteristic zero (by Hironaka's resolution of singularities), and also if $k$ has arbitrary characteristic and $X$ is a curve (by Lipman's resolution of singularities for excellent schemes of dimension two). Most experts believe that it should exist in general, but at this moment, resolution of singularities in positive and mixed characteristic remains one of the big open problems in algebraic geometry.
      Corollary \ref{cor-sk} implies in particular that the homotopy type of the dual complex $|\Delta(\cX_k)|$ does not depend on the choice of the proper $sncd$-model  $\cX_k$.
    This is an analog of Thuillier's generalization of Stepanov's theorem in \cite{thuillier}, saying that the homotopy type of the dual complex of a log resolution of a pair of algebraic varieties over a perfect field is independent of the choice of the log resolution.

\sss \label{sss-defret} If $\overline{X}$ is a smooth compactification of $X$ and $\cX$ is a proper $sncd$-model of $\overline{X}$, then the explicit construction of the strong deformation retract $H$ from Theorem \ref{thm-sk} shows that it restricts to strong deformation retracts of $X^{\an}$ and $X^{\bir}$ onto $\Sk(\cX)$ (in fact, $H(x,t)$ lies in $X^{\bir}$ for every $x$ in $\overline{X}^{\an}$ and every $t>0$). Thus the inclusions $X^{\bir}\to X^{\an}$ and $X^{\an}\to \overline{X}^{\an}$ are homotopy equivalences. In particular, the homotopy type of the analytification of a smooth $K$-variety is a birational invariant if $k$ has characteristic zero.

 \subsection{The deformation retraction in a basic example.}\label{subsec-exam}
\sss We will give an explicit construction of the map $H$ from Theorem \ref{thm-sk} for the following elementary example:
 $$\cX=\Spec R[T_1,T_2]/(T_1^{N_1}T_2^{N_2}-\pi),$$ with $\pi$ a uniformizer in $R$ and $N_1$, $N_2$ positive integers. Then
 $\cX$ is an $sncd$-model for its generic fiber $X=\cX_K$, and
 $\widehat{\cX}_\eta$ is the set of points $x$ in $X^{\an}$ such that  $|T_1(x)|\leq 1$ and $|T_2(x)| \leq 1$.
 We denote by $E_i$ the component of $\cX_k$ defined by $T_i=0$ for $i=1,2$ and by $O$ the unique intersection point of $E_1$ and $E_2$.
 The dual complex $|\Delta(\cX_k)|$ is the standard $1$-simplex
  $$\Delta_1=\{(\lambda,1-\lambda)\in \R^2\,|\,0\leq \lambda \leq 1\}$$ and the morphism $\Phi$ constructed in \eqref{sss-phi}
  sends $(1,0)$ to the divisorial point associated with $(\cX,E_1)$, $(0,1)$ to the divisorial point associated with $(\cX,E_2)$, and
  $(\lambda,1-\lambda)$ to the monomial point associated with
  $$(\cX,(E_1,E_2),(\frac{\lambda}{N_1},\frac{1-\lambda}{N_2}), O)$$
  for all $\lambda\in \,]0,1[$.

\sss   The construction of the map $H$ is best understood in terms of torus actions.
 We set $c=\gcd(N_1,N_2)$ and $M_i=N_i/c$ for $i=1,2$, and we choose integers $a_1$ and $a_2$ such that $a_1M_1+a_2M_2=1$.
  For every complete valued field extension $(L,|\cdot|_L)$ of $K$ we set
 $$\G_L=\{x\in (\Spec L[U_1,U_2]/(U^{M_1}_1U^{M_2}_2-1))^{\an}\,|\ |U_1(x)|=1\}$$
  with the group structure given by componentwise multiplication. For every
  element $t$ in the interval $[0,1]$ we define a point $\gamma_L(t)$ in $\G_L$ as follows.
   We will make use of the isomorphism of $L$-algebras
   $$L[V,V^{-1}]\to L[U_1,U_2]/(U_1^{M_1}U_2^{M_2}-1)\colon V\mapsto U_1^{a_2}U_2^{-a_1}$$
    whose inverse is given by $U_1\mapsto V^{M_2}$ and $U_2\mapsto V^{-M_1}$.
  We can write every polynomial $f$ in $L[V]$ as a Taylor expansion
  $$f=\sum_{i\geq 0}c_{i}(V-1)^i$$
  around the point $1$, where the coefficients $c_{i}$ lie in $L$. Then the point $\gamma_L(t)$ is fully determined by the property that
  $$|f(\gamma_L(t))|=\max_{i\geq 0}|c_{i}|_L t^{i}.$$
  In other words, the point $\gamma_L(t)$ is the sup-norm on the closed disc of radius $t$ around the point $1$ in a completed algebraic closure $\widehat{L^a}$ of $L$. Note that
  $$|U_1(\gamma_L(t))|=|V^{M_2}(\gamma_L(t))|=1$$
   for every $t$ so that $\gamma_L(t)$ is indeed a point of $\G_L$. The map
  $$\gamma_L:[0,1]\to \G_L:t\mapsto \gamma(t)$$
  is a continuous path from $\gamma_L(0)=1$ to $\gamma_L(1)$.

\sss The torus $\G_K$ acts on $\widehat{\cX}_\eta$ by componentwise multiplication, and we can use this action together with the paths $\gamma_L$ to produce paths in $\widehat{\cX}_\eta$. For every point $x$ of $\widehat{\cX}_\eta$, the  action of $\G_K$ gives rise to a continuous map
 $$\G_{\cH(x)}\to (\widehat{\cX}_\eta)\times_K \cH(x):g\mapsto g\cdot x.$$
  For every $t$ in $[0,1]$, we define $H(x,t)$ as the image of $\gamma_{\cH(x)}(t)\cdot x$ under the projection map
  $$(\widehat{\cX}_\eta)\times_K \cH(x)\to \widehat{\cX}_\eta.$$
  In this way, we obtain a map
  $$H\colon \widehat{\cX}_\eta\times [0,1]\to \widehat{\cX}_\eta\colon (x,t)\mapsto H(x,t).$$
 The map $H$ is continuous by continuity of the paths $\gamma_L$ and of the torus action on $\widehat{\cX}_\eta$.

 \sss We can also give a more explicit and down-to-earth (but less conceptual) description of the map $H$. Let $x$
 be a point of $\widehat{\cX}_\eta$ and let $t$ be an element of $[0,1]$. For notational convenience, we set $x_1=T_1(x)$ and $x_2=T_2(x)$; these are elements of the residue field
 $\cH(x)$ of $X^{\an}$ at $x$. Let $f$ be an element of $K[T_1,T_2]$. Then we can write the Laurent polynomial $f(x_1V^{M_2},x_2V^{-M_1})$ in $\cH(x)[V,V^{-1}]$ as a  rational function
$$f(x_1V^{M_2},x_2V^{-M_1})=\frac{1}{V^j}\sum_{i\geq 0}c_{i}(V-1)^i,$$
where the coefficients $c_i$ belong to the valued field $\cH(x)$ and only finitely many of them are non-zero. The point $H(x,t)$ is fully characterized by the property
$$|f(H(x,t))|=\max_{i}|c_{i}|t^{i}.$$ We remark for later reference that
\begin{equation}\label{eq-bound}
 |f(H(x,t))| \geq |c_{0}| = |f(x)|.
\end{equation}

\sss Now we prove that $H$ is a strong deformation retract onto $\Sk(\cX)$.  Setting $t=0$ we find $|f(H(x,0))|=|f(x)|$ so that $H(x,0)=x$. To compute $H(x,1)$ we use the fact that for every  complete valued extension $(L,|\cdot|_L)$ of $K$, the closed disc with radius one around $1$ in $\widehat{L^a}$ coincides with the closed disc with radius one around $0$, so that
$$|g(\gamma_L(1))|=\max_{i} |c_i|_L $$
for every polynomial $g=\sum_{i\geq 0}c_iV^i$ in $L[V]$. In this way, we see that for every polynomial
$$f=\sum_{i,j\geq 0}c_{ij}T_1^iT_2^j$$ in $K[T_1,T_2]$, we have
$$|f(H(x,1))|=\max_{i,j}|c_{ij}|_K |x_1|^{i}|x_2|^{j},$$ or, in additive notation:
$$v_{H(x,1)}(f)=\min_{i,j}\{v_K(c_{ij})+iv_x(T_1)+jv_x(T_2)\}.$$
 Thus
  we find by using formula \eqref{eq-monexpl} that $H(x,1)$ is the monomial point on $X^{\an}$ associated with
 $$(\cX,(E_1,E_2),(v_x(T_1),v_x(T_2)) ,O).$$ This is precisely the image $\rho_{\cX}(x)$ of $x$ under the retraction $\rho_{\cX}\colon \widehat{\cX}_\eta\to \Sk(\cX)$.
 Finally, we show that $H(x,t)=x$ for all $t$ in $[0,1]$ when $x$ is a point of the skeleton $\Sk(\cX)$, that is, a monomial point with respect to $\cX$.
  Direct computation shows that $|T_1(H(x,t))|=|x_1|$ and $|T_2(H(x,t))|=|x_2|$. Combining the inequality \eqref{eq-bound} with the minimality property
  of monomial valuations in Proposition \ref{prop-mon}, we see at once that $H(x,t)$ must be equal to $x$.

\newpage 
\section{Weight functions and the Kontsevich-Soibelman skeleton}

\subsection{The work of Kontsevich and Soibelman}
\sss In \cite{KS}, Kontsevich and Soibelman proposed a new
interpretation of mirror symmetry based on
non-archimedean geometry over the field of complex Laurent series
$\C((t))$. Their fundamental idea was to encode a part of the geometry of a
one-parameter degeneration of complex Calabi-Yau varieties into a
  topological manifold endowed with a
$\Z$-affine structure with singularities, and to interpret mirror
symmetry as a certain combinatorial duality between such
manifolds. They worked out in detail the case of degenerations of
$K3$-surfaces. Similar ideas were developed by Gross and
Siebert in their theory of {\em toric degenerations}. Gross and
Siebert replaced the use of non-archimedean geometry by methods
from tropical and logarithmic geometry and extended the
results for $K3$ surfaces to higher-dimensional
degenerations \cite{gross}.

\sss An essential ingredient of the construction of Kontsevich and
Soibelman is the following. We denote by $\Delta$ a small disc
around the origin of the complex plane and we set
$\Delta^*=\Delta\setminus \{0\}$. We denote by $t$ a local
coordinate on $\Delta$ centered at $0$. Let $X$ be a smooth
projective family of varieties over $\Delta^*$ and let $\omega$ be
a relative differential form of maximal degree on the family $X\to
\Delta^*$.
 Kontsevich and Soibelman associated to these data a skeleton
 $\Sk(X,\omega)$, which is a topological subspace of the Berkovich
 analytfication of the $\C((t))$-variety obtained from $X$ by base
 change. If $X$ is a family of Calabi-Yau varieties, then we set
 $\Sk(X,\omega)=\Sk(X)$ where $\omega$ is any relative volume form
 on $X$. This definition does not depend on the choice of $\omega$.

\sss Kontsevich and Soibelman proved that $\Sk(X,\omega)$ can
 be explicitly  computed on any strict normal crossings model $\cX$ for
 $X$ over $\Delta$: it is a union of faces of the Berkovich skeleton $\Sk(\cX)$ of the model $\cX$ on which $\omega$ is minimal in a suitable sense. Their proof relied on the Weak Factorization
 Theorem. It is interesting to note that, even though the
 Berkovich skeleton $\Sk(\cX)$ from Section \ref{subsec-berksk} heavily depends on the chosen model $\cX$, the
 Kontsevich-Soibelman skeleton $\Sk(X,\omega)$ only depends on $X$
 and $\omega$. It singles out certain faces of $\Sk(\cX)$
 that must appear in the skeleton of {\em every} strict normal
 crossings model.

 \sss In \cite{MuNi}, Mircea Musta\c{t}\u{a} and the author extended this construction to
 varieties over complete discretely valued fields $K$ of arbitrary
 characteristic, and to pluricanonical forms $\omega$. Our approach does not use the Weak Factorization
 Theorem but only relies on basic computations on valuations and canonical
 sheaves. Moreover, we proved that the skeleton of a Calabi-Yau
 variety over $\C((t))$ is always connected. An interesting gadget
 that appears in our work is the {\em weight function}
 $$\weight_{\omega}:X^{\an}\to \R\cup\{+\infty\}$$ associated to a smooth and proper $K$-variety $X$ and a pluricanonical form $\omega$ on $X$.
  This weight function is
  piecewise affine on the Berkovich skeleton of any
 strict normal crossings model of $X$ and strictly increasing as one
 moves away from the Berkovich skeleton. The Kontsevich-Soibelman
 skeleton is precisely the set of points where $\weight_{\omega}$
 reaches its minimal value; see Section \ref{sec-KSdef}.

\subsection{Log discrepancies in birational geometry}

\sss Our approach is inspired by interesting analogies with some
fundamental invariants in birational geometry. Let $X$ be a smooth
complex variety and $\mathcal{I}$ a coherent ideal sheaf on $X$.
Let $v$ be a divisorial valuation on $X$, that is, a positive real
multiple of the discrete valuation $\mathrm{ord}_E$ on the function field $\C(X)$
associated to a prime divisor $E$ on a normal birational
modification $Y$ of $X$. We denote by $N$ the multiplicity of the
scheme $Z(\mathcal{I}\mathcal{O}_Y)$ along $E$  and by $\nu$ the
multiplicity of $E$ in the relative canonical divisor $K_{Y/X}$.
We set
$$\weight_{\mathcal{I}}(v)=\frac{\nu+1}{N}$$ and we call this positive rational
number the weight of $\mathcal{I}$ at $v$. Then the infimum of the
 values $\weight_{\mathcal{I}}(v)$ at all divisorial valuations $v$
on $X$ is called the log-canonical threshold of the pair
$(X,\mathcal{I})$ and denoted by $\lct(X,\mathcal{I})$. This is
 a measure for the singularities of the zero locus $Z(\mathcal{I})$ of $\mathcal{I}$ on $X$, and one of the most important
 invariants in birational geometry. We refer to \cite{kol} for
 more background.

 \sss It is a fundamental fact that the log-canonical threshold
 of $(X,\mathcal{I})$ can be computed on a single log-resolution
  of $(X,\mathcal{I})$, i.e., a proper birational morphism $h:Y\to
 X$ such that $Y$ is smooth, $h$ is an isomorphism over the complement of
 $Z(\mathcal{I})$, and
 $Z(\mathcal{I}\mathcal{O}_Y)$ is a strict normal crossings divisor
 on $Y$. Namely, we have
 $$\lct(X,\mathcal{I})=\min\{\weight_{\mathcal{I}}(v)\}$$ where
 $v$ runs over the divisorial valuations associated to the prime
 components of $Z(\mathcal{I}\mathcal{O}_Y)$. If this minimum is
 reached on a prime component $E$ of
 $Z(\mathcal{I}\mathcal{O}_Y)$, then we say that $E$ computes the
 log-canonical threshold of $(X,\mathcal{I})$. If we denote by
 $\mathcal{E}$ the union of such prime components $E$, then
 the Connectedness Theorem of Shokurov and Koll\'ar \cite[17.4]{kol} states that
 for every point $x$ of $Z(\mathcal{I})$ and every sufficiently
 small open neighbourhood $U$ of $x$ in $Z(\mathcal{I})$, the
  topological space $h^{-1}(U)\cap \mathcal{E}$ is connected. This
  was the main source of inspiration for our theorem on the
  connectedness of the skeleton of a Calabi-Yau variety over
  $\C((t))$ (Theorem \ref{thm-conn}).

\sss In \cite{BFJ} and \cite{JM}, a function closely related to the
weight function $\weight_{\mathcal{I}}$ was extended from the set
of divisorial valuations on $X$ to the non-archimedean link of
$Z(\mathcal{I})$ in $X$, that is, the analytic space over the
field $\C$ with the trivial absolute value that we obtain by
removing the generic fiber of the $\C$-variety $Z(\mathcal{I})$
from the generic fiber of the formal completion of $X$ along
$Z(\mathcal{I})$. We have made a similar construction to define
weight functions on analytic spaces over discretely valued fields;
 this construction will be explained in Section \ref{sec-weight}.

\subsection{Definition of the Kontsevich-Soibelman skeleton}\label{sec-KSdef}
\sss \label{sss-divweight} Let $X$ be a connected, smooth and proper $K$-variety of dimension $n$, and let $\omega$ be a non-zero $m$-pluricanonical form on $X$, that is, a non-zero element of $\omega_{X/K}^{\otimes m}(X)$.
 Let $\cX$ be a regular $R$-model of $X$, $E$ an irreducible
component of $\cX_k$ and $x$ the divisorial point on $X^{\an}$
associated with $(\cX,E)$. The relative canonical sheaf
$\omega_{\cX/R}$ is a line bundle on $\cX$ that extends the
canonical line bundle $\omega_{X/K}$ on $X$. The differential form
$\omega$ on $X$ defines a rational section of $\omega^{\otimes m}_{\cX/R}$ and
thus a divisor $\mathrm{div}_{\cX}(\omega)$ on $\cX$. We denote by
$N$ the multiplicity of $E$ in $\cX_k$ and by $\nu$ the
multiplicity of $E$ in $\mathrm{div}_{\cX}(\omega)$. We define the
weight of $\omega$ at $x$ by the formula
$$\weight_{\omega}(x)=(\nu+m)/N.$$ This definition only depends on $x$ and not on the choice of $\cX$ and $E$. We define the weight of $X$ with respect
to $\omega$ by
$$\weight_{\omega}(X)=\inf\{\weight_{\omega}(x)\,|\,x\in
X^{\divi}\}\in \R\cup \{-\infty\}.$$

\sss A divisorial point $x$ on $X^{\an}$ is called
{\em $\omega$-essential} if the weight function $\weight_{\omega}$
reaches its minimal value at $x$, that is,
$$\weight_{\omega}(X)=\weight_{\omega}(x).$$ The skeleton $\Sk(X,\omega)$ of
 the pair $(X,\omega)$ is defined as the closure of the
set of $\omega$-essential divisorial points in the space of birational points $X^{\bir}$.
%If  the canonical sheaf $\omega_{X/K}$ is trivial, then we set $\Sk(X)=\Sk(X,\omega)$
%where $\omega$ is any volume form on
%$X$. This definition is independent of the choice of $\omega$,
%since multiplying $\omega$ by an element $\lambda$ of $K^*$ shifts
%the weight function by $v_K(\lambda)$. Note also that $\Sk(X)=\Sk(\omega')$ for any non-zero pluricanonical form on $X$, because the
% space of $m$-pluricanonical forms on $X$ is generated by the $m$-th tensor power of $\omega$.
 It is obvious from the definition that $\Sk(X,\omega)$ is a
birational invariant of the pair $(X,\omega)$, since the spaces
$X^{\bir}$ and $X^{\divi}$ and the weight function
$\weight_{\omega}$ are birational invariants.

%\section{The weight function and the Connectedness Theorem}
\subsection{Definition and properties of the weight function}\label{sec-weight}
\sss Without suitable assumptions on the existence of
resolutions of singularities, we cannot say much more about the
skeleton $\Sk(X,\omega)$; for instance, we cannot prove that
$\Sk(X,\omega)$ is non-empty. Therefore, we will assume from now
on that $k$ has characteristic zero or $X$ is a curve. With the
current state of affairs, these are the cases where resolution of
singularities is known in the form that we need. In particular, it
is known that every proper $R$-model of $X$ can be dominated by a proper
$sncd$-model of $X$.
% that is, a regular proper $R$-model $\cX$
%whose special fiber $\cX_k$ is a divisor with strict normal
%crossings.

\sss In Section \ref{subsec-berksk} we have attached to each $sncd$-model $\cX$ of $X$ its
Berkovich skeleton $\Sk(\cX)$. It was defined as the set of all
birational points on $X$ that are
 monomial with respect to the strict normal crossings divisor
$\cX_k$ on $\cX$. We have shown in Proposition \ref{prop-homeo}  that the skeleton $\Sk(\cX)$ is canonically homeomorphic to the
dual complex $|\Delta(\cX_k)|$ of the strict normal crossings divisor $\cX_k$.

\begin{theorem}[Proposition 4.4.5 in \cite{MuNi}]\label{thm-wt}
There exists a unique smallest function
$$\weight_{\omega}:X^{\an}\to \R\cup \{+\infty\}$$ with the
following properties.
\begin{enumerate}
\item The function $\weight_{\omega}$ is lower semi-continuous.

\item \label{it:wt2} Let $\cX$ be an $sncd$-model for $X$ and let $x$ be a point of the
Berkovich skeleton $\Sk(\cX)$. Let $f$ be a rational function on
$\cX$ such that, locally at $\spe_{\cX}(x)$, we have
$$\mathrm{div}(f)=\mathrm{div}_{\cX}(\omega)+m(\cX_k)_{\red}.$$
Then $$\weight_{\omega}(x)=-\ln |f(x)|.$$ In particular,
$\weight_{\omega}$ is continuous and piecewise affine on
$\Sk(\cX)$, and we get the same value as in \eqref{sss-divweight} on divisorial
points. Moreover, for all $x$ in $\widehat{\cX}_\eta$, we have
$$\weight_{\omega}(x)\geq \weight_{\omega}(\rho_{\cX}(x))$$ with
equality if and only if $x\in \Sk(\cX)$.

\item The restriction of $\weight_{\omega}$ to $X^{\bir}$
is a birational invariant of $(X,\omega)$.
\end{enumerate}
\end{theorem}
\begin{proof}
We only give a rough sketch of the arguments and refer to \cite{MuNi} for details. The formula in \eqref{it:wt2} can be used to extend the weight function
$\weight_{\omega}$ to the set $X^{\mon}$ of monomial points on $X$. Of course, each monomial point will belong to the Berkovich skeleta of several $sncd$-models, and one must show that  the formula does not depend on the choice of an $sncd$-model. Next, one proves that the inequality in \eqref{it:wt2} holds for monomial points.
 When $x$ is any point of $X^{\an}$, one sets
$$\weight_{\omega}(x)=\sup_{\cX} \{\weight_{\omega}(\rho_{\cX})\}$$ where $\cX$ runs through the set of proper $sncd$-models of $X$. Then one can prove that the resulting function
$\weight_{\omega}$ on $X^{\an}$ satisfies all the properties in the statement.
\end{proof}

\subsection{Computation of the Kontsevich-Soibelman skeleton}
\sss \label{sss-propweight} We can use the properties of the weight function in Theorem
\ref{thm-wt} to compute the Kontsevich-Soibelman skeleton
$\Sk(X,\omega)$ on a fixed proper $sncd$-model $\cX$ of $X$.
 The divisorial points are dense in each face of $\Sk(\cX)$ (they are precisely the points with barycentric coordinates in $\Q$).
 Point \eqref{it:wt2} of the theorem
 immediately implies that $\Sk(X,\omega)$ is the subspace of the
compact space $\Sk(\cX)$ consisting of the points where the
continuous function $\weight_{\omega}|_{\Sk(\cX)}$ reaches its
minimal value, because it says that the weight function is strictly increasing if we move away from the skeleton (recall that $\widehat{\cX}_\eta=X^{\an}$ if
 $\cX$ is proper over $R$).
 In particular, $\Sk(X,\omega)$ is a non-empty
compact topological space. We can make this description much more
explicit, as follows.

\sss We write $\cX_k=\sum_{i\in I}N_i E_i$. For each $i\in I$, we
denote by $\nu_i$ the multiplicity of $E_i$ in the divisor
$\mathrm{div}_{\cX}(\omega)$. Recall that each face of $\Sk(\cX)$
corresponds to a connected component $C$ of an intersection
 $E_J=\cap_{j\in J}E_j$ where $J$ is a non-empty subset of $I$. We
 say that the face is $\omega$-essential if
 $$\frac{\nu_j+m}{N_j}=\min\left\{\frac{\nu_i+m}{N_i}\,|\,i\in I\right\}$$ for every $j$ in $J$ and
 $C$ is not contained in the Zariski-closure in $\cX$ of the pluricanonical divisor
 $\mathrm{div}_X(\omega)$ (the divisor of zeroes of
 $\omega$ on the $K$-variety $X$).

 \begin{theorem}[Theorem 4.5.5 in \cite{MuNi}]\label{thm-compsk}
The weight of $X$ with respect to $\omega$ is given by
$$\weight_{\omega}(X)=\min\left\{\frac{\nu_i+m}{N_i}\,|\,i\in I\right\}$$
and the skeleton $\Sk(X,\omega)$ is the union of the
$\omega$-essential faces of $\Sk(\cX)$. In particular, this union only depends on $X$ and $\omega$, and not on the choice of the model $\cX$.
 \end{theorem}
\begin{proof}
This follows easily from the properties of the weight function described in Theorem \ref{thm-wt}\eqref{it:wt2}.
 We have already explained in \eqref{sss-propweight} that $\Sk(X,\omega)$ is the locus of points in $\Sk(\cX)$ where $\weight_{\omega}$ reaches its minimal value. Recall that for every $i\in I$, the value of $\weight_{\omega}$ at the vertex of $\Sk(\cX)$ corresponding to $E_i$ is given by $(\nu_i+m)/N_i$.
  The explicit formula for the weight function on $\Sk(\cX)$ implies that it is
   piecewise affine and concave on every face of $\Sk(\cX)$, and it is affine on a face if and only if the corresponding subvariety $C$ of $\cX_k$
    is not contained in the closure of $\mathrm{div}_X(\omega)$; see \cite[4.5.5]{MuNi} for details.
\end{proof}

\begin{exam}
 Suppose that $R=\C\llbr t\rrbr$ and
that $X$ is a $K3$-surface over $K$. Assume that $X$ has an
$sncd$-model $\cX$ such that $\cX_k$ is reduced and
$\omega_{\cX/R}$ is trivial. Such models play an important role
in the classification of semi-stable degenerations of
$K3$-surfaces by Kulikov \cite{kulikov} and Persson-Pinkham \cite{pers-pink}. They have the
special property that $\Sk(X,\omega)=\Sk(\cX)$ for every volume form $\omega$ on $X$, since all multiplicities
$N_i$ are equal to one, and all $\nu_i$ are equal by the triviality of
$\omega_{\cX/R}$.
\end{exam}

\sss In \cite{MuNi} we proved the following variant of the
Shokurov-Koll\'ar Connectedness Theorem.
 Like the proofs of Shokurov and Koll\'ar, our proof is based on
 vanishing theorems: we proved
 generalizations of Kawamata-Viehweg Vanishing and Koll\'ar's
Torsion-free Theorem for varieties over power series rings in
characteristic zero by means of Greenberg approximation.

\begin{theorem}[Theorem 5.3.3 in \cite{MuNi}]\label{thm-conn}
Assume that the residue field $k$ of $K$ has characteristic zero. If $X$ is a geometrically connected, smooth and proper $K$-variety
of geometric genus one, and $\omega$ is a non-zero canonical form on $X$, then $Sk(X,\omega)$ is connected.
\end{theorem}

 One can say much more using advanced tools from the Minimal Model Program, as we will explain in the following section.

\section{The essential skeleton and the Minimal Model Program}
\subsection{The essential skeleton}
\sss Throughout this section, we assume that the residue field $k$ of $K$ has characteristic zero.
  Let $X$ be a connected, smooth and projective $K$-variety (the projectivity condition is needed to apply results from the Minimal Model Program).
   Let $\cX$ be a proper $sncd$-model of $X$. We have seen in Theorem \ref{thm-compsk} that for every
non-zero pluricanonical form $\omega$ on $X$,  the Kontsevich-Soibelman skeleton $\Sk(X,\omega)$ singles out certain faces of $\Sk(\cX)$ that do not depend on the model
 $\cX$. In \cite[\S4.6]{MuNi} we defined the {\em essential skeleton} $\Sk(X)$ of $X$ as the union of the Kontsevich-Soibelman skeleta $\Sk(X,\omega)$ for all non-zero pluricanonical forms $\omega$ on $X$.

 \sss If $\omega_{X/K}$ is trivial and $\omega$ is a volume form on $X$, then it is not hard to see that $\Sk(X)=\Sk(X,\omega)$: multiplying $\omega$
  with an element $\lambda$ in $K^{\ast}$ simply shifts
the weight function by $v_K(\lambda)$, and for every $m>0$, the
 space of $m$-pluricanonical forms on $X$ is generated by the $m$-th tensor power of $\omega$.
 It is not true for general $X$, however,
  that $\Sk(X)=\Sk(X,\omega)$ for some fixed pluricanonical form $\omega$ on $X$.

\sss Without suitable conditions on $X$, we cannot hope that $\Sk(X)$ is a strong deformation retract of $X^{\an}$. For instance, if $X$ is rational (e.g., a projective space $\mathbb{P}^n_K$) then all pluricanonical forms on $X$ are zero and the essential skeleton is empty. However, Chenyang Xu and the author proved in \cite{NiXu} that
 $\Sk(X)$ is a strong deformation retract of $X^{\an}$ if $X$ has ``enough'' pluricanonical forms. Our proof is based on the Minimal Model Program, and in particular on the results in \cite{dFKX}. We will now briefly explain the main ideas.

\subsection{$dlt$ models}
\sss If $X$ is a curve of genus $\geq 1$, then $X$ has a unique minimal $sncd$-model $\cX$, and thus a canonical Berkovich skeleton $\Sk(\cX)$.
 If $X$ has dimension at least two, however, minimal $sncd$-models no longer exist in general. In order to get a good notion of minimal model, we have to enlarge the class
  of $sncd$-models to so-called (good) $dlt$-models. The abbreviation $dlt$ stands for {\em divisorially log terminal}. The precise definition of a $dlt$-model $\cX$ is quite technical; we do not give it here but refer to \cite[\S2.1]{NiXu} instead. The basic idea is that we allow certain mild singularities on $\cX$, in accordance with the general philosophy of the Minimal Model Program. The set of points of $\cX$ where $\cX$ is regular and $\cX_k$ is a strict normal crossings divisor is an open subscheme of $\cX$ that we denote by $\cX^{\sncd}$. The definition of a $dlt$-model guarantees that $\cX^{\sncd}$ is still sufficiently large to capture all the important information about the skeleton; we set $$\Sk(\cX):=\Sk(\cX^{\sncd})\subset X^{\bir}.$$

  \sss A $dlt$-model $\cX$ of $X$ is called minimal if the line bundle $\omega_{\cX/R}((\cX_k)_{\red})$ is semi-ample, which means that some power of this line bundle  is generated by global sections. Fundamental theorems in birational geometry imply that $X$ has a minimal $dlt$-model if and only if the canonical line bundle $\omega_{X/K}$ is semi-ample (we refer to \cite[2.2.6]{NiXu} for detailed references).  To be precise, we should assume that $X$ is defined over an algebraic curve because the necessary tools from the Minimal Model Program have only been developed under that assumption, but we will ignore this issue here; if $\omega_{X/K}$ is trivial one can get rid of the algebraicity condition by using tools from logarithmic geometry \cite[\S4.2]{NiXu}.

\sss    Minimal $dlt$-models are not unique, but they are closely related (birationally crepant) and the skeleton $\Sk(\cX)$ does not depend on the choice of the minimal $dlt$-model $\cX$. One of the main results in \cite{NiXu} is the following.

\begin{theorem}[Theoren 3.3.4 in \cite{NiXu}]
If $\omega_{X/K}$ is semi-ample and $\cX$ is any minimal $dlt$-model of $X$, then the skeleton $\Sk(\cX)$ is equal to the essential skeleton $\Sk(X)$ of $X$.
\end{theorem}

%\begin{exam}
%Assume that $X$ is a geometrically connected smooth projective curve over $K$. Then $\omega_{X/K}$ is semi-ample if and only if the genus of $X$ is at least one. In that case, the %minimal $sncd$-model $\cX$ of
%$X$ is a minimal $dlt$-model. Thus the essential skeleton $\Sk(X)$ of $X$ is precisely $\Sk(\cX)$.
%\end{exam}
% Klopt dit wel? Moeten we niet alle rationale "staarten" samentrekken?

\sss If $\cX'$ is any proper $sncd$-model of $X$, then the Minimal Model Program tells us how to modify $\cX'$ into a minimal $dlt$-model $\cX$ by a series of divisorial contractions and flips.
 The effect of the steps in  the Minimal Model Program on the Berkovich skeleton $\Sk(\cX')$ was carefully studied in \cite{dFKX}, and these authors proved that
 $\Sk(\cX)$ can be obtained from $\Sk(\cX')$ by means of a sequence of {\em elementary collapses}, combinatorial operations on simplicial complexes which are, in particular, strong deformation retracts. Since we already know that $\Sk(\cX')$ is a strong deformation retract of $X^{\an}$ by Theorem \ref{thm-sk}, we obtain the following result.

\begin{theorem}[Corollary 3.3.6 in \cite{NiXu}]\label{thm-defret}
If $\omega_{X/K}$ is semi-ample then the essential skeleton $\Sk(X)$ is a strong deformation retract of $X^{\an}$.
\end{theorem}

\sss Under certain conditions on $X$, one can use further results from the Minimal Model Program to obtain information about the topological properties of
 $\Sk(X)$. For instance, if $\omega_{X/K}$ is trivial, the residue field $k$ is algebraically closed and the skeleton $\Sk(X)$ has maximal dimension (that is, the same dimension as $X$) then results by Koll{\'a}r and Kov{\'a}cs imply that $\Sk(X)$ is a closed pseudo-manifold (see \cite[4.1.7 and 4.2.4]{NiXu}).

\sss It would be quite interesting to have a proof of Theorem \ref{thm-defret} that does not make use of the Minimal Model Program and the arguments in \cite{dFKX}, but instead uses the properties of the weight function from Section \ref{sec-weight} and the geometric structure of the Berkovich space $X^{\an}$. A possible approach is the following. Assume that $\omega_{X/K}$ is trivial and let $\omega$ be a volume form on $X$. Then the essential skeleton $\Sk(X)$ of $X$ is the locus where the weight function $\weight_{\omega}$ on $X^{\an}$ reaches its minimal value, and we have seen in Theorem \ref{thm-wt} that it is strictly increasing if one moves away from the Berkovich skeleton of any $sncd$-model of $X$. It is tempting to speculate that one can attach a gradient vector field on $X^{\an}$ to $\weight_{\omega}$ that induces a flow that contracts $X^{\an}$ onto $\Sk(X)$. We refer to \cite{NiXu-poles} for partial results in this direction and applications to the study of motivic zeta functions.

\end{document}